\documentclass[11pt]{amsart}
\usepackage{a4wide}
\usepackage[T1]{fontenc}
\usepackage{amssymb,amsmath,amsthm,latexsym}
\usepackage{mathrsfs}
\usepackage[usenames,dvipsnames]{color}
\usepackage{euscript}
\usepackage{graphicx}
\usepackage{mdwlist}
\usepackage{enumerate}
\usepackage[numbers]{natbib} 
\usepackage{mathtools,dsfont,wasysym}
\usepackage{bbm}
\usepackage{stmaryrd}
\usepackage{centernot}
\usepackage{todonotes}
\usepackage{hyperref}
\hypersetup{colorlinks=true,  linkcolor=blue, citecolor=magenta, urlcolor=cyan}

\makeatletter
\@namedef{subjclassname@2020}{\textup{2020} Mathematics Subject Classification}
\makeatother

\newtheorem{thm}{Theorem}[section]
\newtheorem{cor}[thm]{Corollary}
\theoremstyle{remark}
\newtheorem{rmk}[thm]{Remark}
\theoremstyle{definition}
\newtheorem{defn}[thm]{Definition}
\newtheorem{exm}[thm]{Example}
\numberwithin{equation}{section}
\makeatother

\newcounter{smallromans}

{\end{list}}

\begin{document}
\title[Nonexistence results for Parabolic Problems on Weighted Graphs]{NONEXISTENCE RESULTS FOR A GENERAL
CLASS OF PARABOLIC PROBLEMS WITH A
POTENTIAL ON WEIGHTED GRAPHS}

\author[]{}
\address[]{}
\email{}

\author[D.~von~Criegern]{Dorothea-Enrica von Criegern}
\address[D.~von~Criegern]{Politecnico di Milano, Dipartimento di Matematica, Piazza Leonardo da Vinci 32, 20133 Milano, Italy}
\email{dorotheaenrica.von@polimi.it}

\date{\today}

\begin{abstract} 
We establish nonexistence conditions for nonnegative nontrivial solutions to a class of semilinear parabolic equations with a positive potential on weighted graphs, extending results in \cite{dariopara} to a broader setting that includes both the Porous Medium Equation and the Fast Diffusion Equation. We identify conditions related to the graph’s geometry, the potential’s behaviour at infinity, and bounds on the Laplacian of the distance function under which nonexistence holds. Using a test function argument, we derive explicit parameter ranges for nonexistence.
\end{abstract}
\maketitle


\section{Introduction}
We investigate the nonexistence of nonnegative, nontrivial global solutions to the following parabolic semilinear inequality:
\begin{equation}\label{maineq}
    u_t \geq \Delta (F(u)) + v u^{\sigma} \quad \text{in } V \times (0, +\infty),
\end{equation}
where $V$ is a weighted graph, $\Delta$ denotes the weighted Laplacian on the graph (see Sections~\ref{setting} and ~\ref{lapsetting} for definitions and preliminaries), $v$ is a positive potential, and $\sigma > 1$. The function $F \colon [0, +\infty) \to [0, +\infty)$ satisfies $F(p) \leq C p^m$ for some $C > 0$, $m \in (0, +\infty)$, and all $p \in [0, +\infty)$.

A special case of \eqref{maineq} is the \textit{Generalised Porous Medium Equation} (GPME), also known as the \textit{Filtration Equation}, which arises when $F$ is strictly monotone increasing and satisfies $F(0) = 0$. When $F(u) = u^m$, the equation reduces to two well-known nonlinear generalisations of the heat equation: the \textit{Porous Medium Equation} (PME) for $m > 1$ and the \textit{Fast Diffusion Equation} (FDE) for $m \in (0,1)$. Remarkably, our results apply to both cases, even though they exhibit fundamentally different behaviours -- such as finite versus infinite speed of propagation and finite-time extinction phenomena in the FDE, see \cite{vazquez2007porous}. For investigations of the PME and FDE on Riemannian manifolds, see, e.g., \cite{fde1, grillo2025porous, fde2, meglioli2020blow}.

The (non)existence of solutions to PDEs of the form in $\eqref{maineq}$ has been extensively studied in $\mathbb{R}^N$. Fujita \cite{fujita1966blowing} analysed global existence and blow-up for
\begin{equation}\label{intropara}
    u_t \geq \Delta u + u^{\sigma} \quad \text{in } \mathbb{R}^N \times (0, +\infty).
\end{equation}
In \cite{fujita1966blowing} and later contributions \cite{hayakawa1973nonexistence, kobayashi1977growing}, it was shown that for $1 < \sigma \leq \sigma^* \coloneqq 1 + \frac{2}{N}$, no nonnegative, nontrivial global solutions exist, while for $\sigma > \sigma^*$, solutions corresponding to nonnegative initial data that are sufficiently small in a suitable sense exist globally in time. For more results on blow-up and (global) (non)existence of solutions to semilinear parabolic equations in $\mathbb{R}^N$, see, e.g.,  \cite{bandle1998blowup, deng2000role, levine1990role, pohozaev2000blow}, and references therein.

On Riemannian manifolds, the behaviour of solutions depends significantly on the geometry and volume growth of the manifold. For instance, in \cite{mastrolia2017nonexistence}, the nonexistence of nontrivial global solutions to the problem corresponding to \eqref{intropara} was established under the following weighted space-time volume growth condition for geodesic balls $B_R$:
\begin{equation}\label{volgrmanifold}
    \mathrm{Vol}(B_R)\leq CR^{\frac{2}{\sigma-1}}(\log R)^{\frac{1}{\sigma-1}}\quad \text{for all }R>0\text{ sufficiently large.}
\end{equation}
Other relevant results on Riemannian manifolds include \cite{bandle2011fujita,grigrm1,gu2020global, punzo2012blow, punzo2021global,wang2016asymptotic, zhang1999blow}.

Recently, the problem in $\eqref{intropara}$ and its elliptic counterpart on infinite, connected, locally finite graphs (see Section~\ref{setting} for a definition) have attracted a lot of attention, see, e.g., \cite{biagi2024liouville,grillo2024blow,gu2023semi,lin2017existence,yong2018blow,lin2021heat,darioelliptic,dariopara}, and references therein. In \cite{gu2023semi}, Gu, Huang and Sun studied the problem
\begin{equation}\label{ellpb}
    0 \geq \Delta u + v u^{\sigma} \quad \text{in } V,
\end{equation}
with $v \equiv 1$, assuming the existence of $p_0 > 0$ such that for all $x \sim y \in V$,
\begin{equation}\label{paper1cond}
    \frac{\omega_{xy}}{\mu(x)} \geq \frac{1}{p_0}
\end{equation}
(see Section~\ref{setting} for notation). For $\sigma > 1$, they employed a test function argument to establish nonexistence under the assumption in $\eqref{paper1cond}$ and a volume growth condition with respect to the natural distance (see $\eqref{natdist}$ for a definition), in the spirit of $\eqref{volgrmanifold}$. Similar results were obtained in \cite{hao2022sharpliouvilletyperesults} for the semilinear elliptic inequality
\begin{equation*}
\Delta u + u^p | \nabla u |^q \leq 0, \quad (p,q) \in \mathbb{R}^2.
\end{equation*}

In \cite{darioelliptic}, Monticelli, Punzo, and Somaglia studied $\eqref{ellpb}$ on weighted graphs with a positive potential $v$, removing the condition in $\eqref{paper1cond}$; instead, they assumed an upper bound on $\Delta d(\cdot,x_0)$, for some fixed $x_0 \in V$ and a general pseudometric $d$. The result in \cite{gu2023semi} permits faster volume growth of balls in the graph, whereas the technique in \cite{darioelliptic} can be directly extended to the parabolic counterpart, i.e., $\eqref{maineq}$ with $F(u) = u$. This is done in \cite{dariopara}, where nonexistence results for nontrivial nonnegative solutions were established through a priori estimates on the solution $u$, using a test function argument and growth constraints on the potential at infinity.

Further results on nonexistence, and blow-up phenomena for parabolic problems on graphs can be found in \cite{grillo2024blow,lin2017existence,yong2018blow,lin2021heat}. These works employ techniques such as the semigroup approach and heat kernel estimates, which differ entirely from those used in \cite{dariopara}.

In this paper, we extend the results of \cite{dariopara} to the broader setting in \eqref{maineq} to include nonlinear functions $F$. This extension encompasses important cases such as the PME and FDE, and provides a unified framework for analysing a broader class of nonlinear functions $F$. Additionally, we establish a nonexistence result on finite graphs for all functions $ F \colon [0,+\infty) \to \mathbb{R}$ in $\eqref{maineq}$, regardless of their specific form. 

The paper is structured as follows. In Section $2$, we introduce the mathematical framework and present the main results. Section $3$ is dedicated to the proofs of the main theorem and its corollaries. In Section $4$, we provide examples to highlight some implications of our results. Finally, Section $5$ briefly explores the case of finite graphs.

\bigskip

\noindent{\bf Acknowledgements.} The author would like to thank Dario D. Monticelli for his support and insightful discussions. The author is a member of GNAMPA (Gruppo Nazionale per l'Analisi Matematica, la Probabilità e le loro Applicazioni) of INdAM.

\section{Assumptions and Main Results}
\subsection{Graph Setting}\label{setting}
A \textit{weighted graph} $(V, \omega, \mu)$ consists of a countable set $V$, a \textit{node measure} $\mu\colon V \to (0, +\infty)$, a symmetric \textit{edge weight} function $\omega\colon V \times V \to [0, +\infty)$ satisfying: 
\begin{equation*}
\begin{aligned}
  & (i) \; \omega_{xy} = \omega_{yx} \;\;\; \text{for all} \; x, y \in V \\
  & (ii) \; \omega_{xx} = 0 \quad \text{for all} \;\;\; x \in V \; (\text{i.e., no loops}) \\
  & (iii) \; \sum_{y \in V} \omega_{xy} < \infty \;\;\; \text{for all} \; x \in V
\end{aligned}
\end{equation*}

Let $x, y \in V$, then: 
\begin{itemize}
    \item $x$ is \textit{connected} to $y$ (denoted $x \sim y$) if and only if $\omega_{xy} > 0$;  
    \item the pair $(x, y)$ is an \textit{edge} with endpoints $x, y$ if and only if $x\sim y$, and $E$ denotes the set of edges;
    \item a \textit{path} in $V$ is a sequence of vertices $\{x_k\}_{k=0}^n$ such that $x_k \sim x_{k+1}$ for $k = 0, \dots, n-1$. 
\end{itemize}

A weighted graph $(V, \omega, \mu)$ is \textit{locally finite} if each vertex has finitely many neighbours, \textit{connected} if any two distinct vertices are joined by a path, \textit{undirected} if edges have no orientation.  

The \textit{degree} of $x \in V$ is the cardinality of $\{y \in V : y \sim x\}$. A \textit{pseudo metric} on $V$ is a symmetric map $d\colon V \times V \to [0, +\infty)$ with $d(x, x) = 0$ that satisfies the triangle inequality. In general, $d$ does not satisfy the definition of a metric, as there can exist distinct nodes $x, y \in V, x\neq y$ for which $d(x, y) = 0$. The \textit{jump size} $j > 0$ of $d$ is defined as:  
\begin{equation*}  
    j \coloneqq \sup\{d(x, y) : x, y \in V, \omega_{xy} > 0\}.  
\end{equation*}  
For $\Omega \subset V$, its \textit{volume} is
\begin{equation*}
    \text{Vol}(\Omega) \coloneqq \sum_{x \in \Omega} \mu(x),
\end{equation*}
and $1_\Omega$ denotes its characteristic function.

\subsection{The Weighted Laplacian}\label{lapsetting}
Let $\mathcal{F}$ denote the set of all functions $f\colon V \to \mathbb{R}$. For any $f \in \mathcal{F}$ and $x, y \in V$, we introduce the following

\begin{defn}  
Let $(V, \omega, \mu)$ be a weighted graph. For any $f \in \mathcal{F}$, we define:  
\begin{itemize}  
    \item The \textit{difference operator}  
    \begin{equation*}  
        \nabla_{xy} f \coloneqq f(y) - f(x) \quad \text{for all } x, y \in V;  
    \end{equation*}  
    \item The \textit{(weighted) Laplace operator} 
    \begin{equation*}  
        \Delta f(x) \coloneqq \frac{1}{\mu(x)} \sum_{y \in V} \omega_{xy} [f(y) - f(x)] = \frac{1}{\mu(x)} \sum_{y \sim x} \omega_{xy} [f(y) - f(x)] \quad \text{for all } x \in V.  
    \end{equation*}  
\end{itemize}  
\end{defn}  

One can easily show that for any $f, g \in \mathcal{F}$:  
\begin{itemize}  
    \item The \textit{product rule} holds:  
    \begin{equation*}  
        \nabla_{xy} (fg) = f(x) (\nabla_{xy} g) + (\nabla_{xy} f) g(y) \quad \text{for all } x, y \in V;  
    \end{equation*}  
    \item The \textit{integration by parts formula} holds:  
    \begin{equation}\label{ip}  
        \sum_{x \in V} [\Delta f(x)] g(x) \;\mu(x) = -\frac{1}{2} \sum_{x, y \in V} \omega_{xy} (\nabla_{xy} f) (\nabla_{xy} g) = \sum_{x \in V} f(x) [\Delta g(x)] \;\mu(x),  
    \end{equation}  
    if at least one of the functions has finite support.  
\end{itemize}  
\subsection{Main Results}  
\begin{defn}  
We say that $u\colon V \times [0, +\infty) \to \mathbb{R}$ is a \textit{nonnegative global very weak solution} of \eqref{maineq} if $u \geq 0$ and  
\begin{equation}\label{defneq}  
\begin{split}  
    \int_0^{\infty} \sum_{x \in V} & \Big[\Delta(F(u(x,t))) \phi(x,t) +  v(x,t) u^{\sigma}(x,t) \phi(x,t)\\ & +  u(x,t) \phi_t(x,t)\Big] \;\mu(x)\;dt  
    + \sum_{x \in V} u(x,0) \phi(x,0) \;\mu(x) \leq 0  
\end{split}  
\end{equation}  
holds for all $\phi\colon V \times [0, +\infty) \to \mathbb{R}$ satisfying $\phi \geq 0$, $\text{supp}\, \phi \subset [0, T] \times A$ for some $T > 0$ and finite $A \subset V$, and $\phi(x, \cdot) \in W^{1,1}([0, +\infty))$ for each $x \in V$.
\end{defn}  

\begin{rmk}  
If $u(x, \cdot) \in W^{1,1}_{\text{loc}}([0, +\infty))$ for every $x \in V$, then \eqref{defneq} follows by testing \eqref{maineq} against a test function $\phi$ through integration by parts.  
\end{rmk}  
In this paper, we will always make the following 

\noindent{\underline{\textbf{Assumption} \textbf{(A)}}}  
\begin{itemize}  
    \item[(i)] $(V, \omega, \mu)$ is a connected, locally finite, undirected, weighted graph;  
    \item[(ii)] There exists a constant $C > 0$ such that for every $x \in V$, 
    \vspace{0.1in}
    \begin{equation*}  
    \hspace*{-4in} 
        \sum_{x \sim y} \omega_{xy} \leq C \mu(x);  
    \end{equation*}  
    \item[(iii)] There exists a pseudo metric $d$ with finite jump size $j$;  
    \item[(iv)]  For any $x \in V$ and $r > 0$, the ball $B_r(x)$ with respect to $d$ is finite; 
    \item[(v)] For some $x_0 \in V$, $R_0 > 1$, $\alpha \in [0, 1]$, and $C > 0$,  
    \vspace{-0.05in}
    \begin{equation*}\label{distbound} 
    \hspace*{-2in} 
        \Delta d(x, x_0) \leq \frac{C}{d^{\alpha}(x, x_0)} \quad \text{for any } x \in V \setminus B_{R_0}(x_0).  
    \end{equation*}  
\end{itemize}  

\begin{rmk} \label{natdistrmk}
If conditions (i)--(iv) in Assumption \textbf{(A)} are satisfied, then condition (v) follows directly with $\alpha = 0$; see Remark 2.4 in \cite{dariopara}. In particular, this applies to graphs satisfying (i) and (ii) that are equipped with the natural distance $d_*$, defined by
\begin{equation}\label{natdist}
    d_*(x, y) \coloneqq \min \{ k \in \mathbb{N} : \exists \text{ a path of } k+1 \text{ nodes connecting } x \text{ and } y \}  
\end{equation}  
for every $x, y \in V$. Note that the corresponding jump size is $j_* = 1$.  

When $\alpha>0$ in (v), $\Delta d(x,x_0)$ decays faster, relaxing the growth condition in Theorem~\ref{mainthm} below.
\end{rmk}  
\begin{thm}\label{mainthm}
Let Assumption \textbf{(A)} hold and assume the graph $(V, \omega, \mu)$ is infinite. Let $v\colon V \times [0, +\infty) \to \mathbb{R}$ be a positive function. Fix $\sigma > \max(1, m)$, where $m \in (0, +\infty)$ is a constant, and $\alpha \in [0, 1]$ as in \textbf{(A)}. Suppose there exists $C > 0$ such that $0\leq F(p) \leq C p^m$ for all $p\geq0$. Additionally, assume there are constants $\theta_1 \geq 2$ and $\theta_2 \geq 1$ such that for every $R \geq R_0 > 1$, the following holds:

\begin{equation}\label{hp1}
    \int_0^{\infty} \sum_{x \in V} v^{-\frac{1}{\sigma-1}}(x,t) 1_{E_R}(x,t) \;\mu(x)\;dt \leq C R^{\frac{\theta_1 \sigma}{\theta_2 (\sigma-1)}}  
\end{equation}  
and  
\begin{equation}\label{hp2}
    \int_0^{\infty} \sum_{x \in V} v^{-\frac{m}{\sigma-m}}(x,t) 1_{E_R}(x,t) \;\mu(x)\;dt \leq C R^{(1+\alpha) \frac{\sigma}{\sigma-m}},  
\end{equation}  
where  
\begin{equation}\label{er}
    E_R \coloneqq \{ (x,t) \in V \times [0, +\infty) : R^{\theta_1} \leq d(x_0, x)^{\theta_1} + t^{\theta_2} \leq 2R^{\theta_1} \},  
\end{equation}  
with $x_0$ as in \textbf{(A)}. If $u\colon V \times [0, +\infty) \to \mathbb{R}$ is a nonnegative global very weak solution of \eqref{maineq}, then $u \equiv 0$.
\end{thm}
We now discuss some consequences of Theorem~\ref{mainthm}, where the potential $v$ takes a special form.
\begin{cor}\label{cor1}
Let Assumption \textbf{(A)} hold and assume that the graph $(V, \omega, \mu)$ is infinite. Let $v\colon V \times [0, +\infty) \to \mathbb{R}$ be a positive function such that $v(x,t) \geq f(t) g(x)$ for every $(x,t) \in V \times [0, +\infty)$, where $f$ and $g$ are positive functions. Let $\sigma > \max(1,m)$, where $m \in (0, +\infty)$ is fixed. Suppose there exists a constant $C > 0$ such that $0\leq F(p) \leq C p^m$ for all $p\geq0$. Further assume that for every $R \geq R_0 > 1$, $T \geq T_0 > 0$, and $x_0$ as in \textbf{(A)},
\begin{equation}\label{cor1a1}  
    \int_0^T f^{-\frac{1}{\sigma-1}}(t)  \;dt  \leq C T^{\delta_1} \quad \text{and} \quad \sum_{x \in B_R(x_0)}  g^{-\frac{1}{\sigma-1}}(x) \;\mu(x)\leq C R^{\delta_2},  
\end{equation}  
and  
\begin{equation}\label{cor1a2}  
    \int_0^T f^{-\frac{m}{\sigma-m}}(t) \;dt  \leq C T^{\delta_3} \quad \text{and} \quad \sum_{x \in B_R(x_0)} g^{-\frac{m}{\sigma-m}}(x) \;\mu(x) \leq C R^{\delta_4},  
\end{equation}  
where $\delta_1, \delta_2, \delta_3, \delta_4 \geq 0$ satisfy, with $\alpha$ as in \textbf{(A)}:
\begin{itemize}
    \item[(i)] $0 \leq \delta_1 \leq \frac{\sigma}{\sigma-1}$ and $0 \leq \delta_4 \leq (1+\alpha)\frac{\sigma}{\sigma-m}$;
    \item[(ii)] $\delta_1 = \frac{\sigma}{\sigma-1} \implies \delta_2 = 0$; $\delta_4 = (1+\alpha)\frac{\sigma}{\sigma-m} \implies \delta_3 = 0$;
    \item[(iii)] $\delta_2, \delta_3 \neq 0 \implies 0 \leq \delta_2\delta_3 \leq \left(\frac{\sigma}{\sigma-1} - \delta_1\right)\left((1+\alpha)\frac{\sigma}{\sigma-m} - \delta_4\right)$.
\end{itemize}
If $u\colon V \times [0, +\infty) \to \mathbb{R}$ is a nonnegative global very weak solution of $\eqref{maineq}$, then $u \equiv 0$.
\end{cor}
\begin{cor}\label{cor2}
Let Assumption \textbf{(A)} hold and assume that the graph $(V, \omega, \mu)$ is infinite. Let $v\colon V \times [0, +\infty) \to \mathbb{R}$ be a positive function such that $v(x,t) \geq g(x)$ for every $(x,t) \in V \times [0, +\infty)$, where $g$ is a positive function. Let $\sigma > \max(1,m)$, where $m \in (0, +\infty)$ is fixed. Suppose there exists a constant $C > 0$ such that $0\leq F(p) \leq C p^m$ for all $p\geq0$. Further assume that for every $R \geq R_0 > 1$ and $x_0$ as in \textbf{(A)}, 
\begin{equation*}
    \sum_{x \in B_R(x_0)}  g^{-\frac{1}{\sigma-1}}(x) \;\mu(x)\leq C R^{\delta} \quad \text{and} \quad \sum_{x \in B_R(x_0)}  g^{-\frac{m}{\sigma-m}}(x) \;\mu(x)\leq C R^{\delta'},  
\end{equation*}  
where $\delta,\delta'\geq0$ satisfy, with $\alpha$ as in \textbf{(A)}:
\begin{equation*}
    0\leq \delta(\sigma-1)+\delta'\leq(1+\alpha) \frac{\sigma}{\sigma-m}.
\end{equation*}
If $u\colon V \times [0, +\infty) \to \mathbb{R}$ is a nonnegative global very weak solution of \eqref{maineq}, then $u \equiv 0$.  
\end{cor}

\begin{cor}\label{cor3}
Let Assumption \textbf{(A)} hold and assume that the graph $(V, \omega, \mu)$ is infinite. Let $v\colon V \times [0, +\infty) \to \mathbb{R}$ be a positive function such that $v(x,t) \geq C'$ for some constant $C'>0$ and every $(x,t) \in V \times [0, +\infty)$. Let $\sigma > \max(1,m)$, where $m \in (0, +\infty)$ is fixed. Suppose there exists a constant $C > 0$ such that $0\leq F(p) \leq C p^m$ for all $p\geq0$. Further assume that for every $R \geq R_0 > 1$, and $\alpha,x_0$ as in \textbf{(A)},   
\begin{equation*}  
    \text{Vol}(B_R(x_0)) \leq C R^{\frac{1+\alpha}{\sigma-m}},  
\end{equation*}  
If $u\colon V \times [0, +\infty) \to \mathbb{R}$ is a nonnegative global very weak solution of \eqref{maineq}, then $u \equiv 0$. 
\end{cor}
\section{Proofs of the Main Results}

\begin{proof}[Proof of Theorem~\ref{mainthm}] 
Let $\phi \in \mathcal{C}^2([0, +\infty))$ satisfy  
\begin{equation*}  
    \phi = \begin{cases}  
        1 & \text{in } [0, 1], \\  
        0 & \text{in } [2, +\infty).  
    \end{cases}  
\end{equation*}  
For $x \in V$, $t \in [0, +\infty)$, and $R \geq R_0$ (with $R_0$ from Assumption \textbf{(A)}) define  
\begin{equation*}  
    \psi_R(x,t) \coloneqq \frac{d(x, x_0)^{\theta_1}+t^{\theta_2} }{R^{\theta_1}}  \;\;\;\text{and}\;\;\;
    \phi_R(x,t) \coloneqq \phi(\psi_R(x,t)).  
\end{equation*}  
We will use the following upper bounds:  
\begin{itemize}  
    \item[(i)] There exists $C \geq 0$ such that for every $x \in V$ and $t \in [0, +\infty)$,  
    \begin{equation}\label{est1}
        -\Delta \phi_R(x,t) \leq \frac{C}{R^{1+\alpha}} 1_{F_R}(x,t),  
    \end{equation}  
    where  
    \begin{equation*}
        F_R \coloneqq \{ (x,t) \in V \times [0, +\infty) : (R/2)^{\theta_1} \leq d(x_0, x)^{\theta_1} + t^{\theta_2} \leq (4R)^{\theta_1} \}.  
    \end{equation*}  

    \item[(ii)] There exists $C \geq 0$ such that for every $x \in V$ and $t \in [0, +\infty)$,  
    \begin{equation}  \label{est2}
        -\partial_t \phi_R(x,t) \leq \frac{C}{R^{\frac{\theta_1}{\theta_2}}} 1_{E_R}(x,t).  
    \end{equation}  
\end{itemize}  
For a detailed proof, see \cite{dariopara}. Here, we provide only a brief summary for the reader's convenience. We first consider (i). Let $x \in V$, $t \geq 0$. By a Taylor expansion for $\phi$, for every $y \sim x$, there exists some $\xi_y$ between $\psi_R(x,t)$ and $\psi_R(y,t)$ such that  
    \begin{equation*}  
    \begin{split}  
        -\Delta \phi_R(x,t) =& -\frac{1}{\mu(x)} \phi'(\psi_R(x,t)) \sum_{y \sim x} \omega_{xy} \left( \frac{d(x_0, y)^{\theta_1} - d(x_0, x)^{\theta_1}}{R^{\theta_1}} \right) \\  
        &- \frac{1}{2\mu(x)} \sum_{y \sim x} \omega_{xy} \phi''(\xi_y) \left( \frac{d(x_0, y)^{\theta_1} - d(x_0, x)^{\theta_1}}{R^{\theta_1}} \right)^2 \\  
    \end{split}  
    \end{equation*}  
By two more Taylor expansions for $p \mapsto p^{\theta_1}$ of first and second order, for every $y \sim x$, there exist $\eta_y, \lambda_y$ between $d(x_0, x)$ and $d(x_0, y)$ such that  
    \begin{equation*}  
    \begin{split}  
        -\Delta \phi_R(x,t) =& -\frac{\theta_1 d(x_0, x)^{\theta_1 - 1} \phi'(\psi_R(x,t))}{R^{\theta_1}} \Delta d(x_0, x) \\  
        &- \frac{\theta_1 (\theta_1 - 1) \phi'(\psi_R(x,t))}{2R^{\theta_1}} \frac{1}{\mu(x)} \sum_{y \sim x} \omega_{xy} \eta_y^{\theta_1 - 2} (d(x_0, y) - d(x_0, x))^2 \\  
        &- \frac{1}{2\mu(x)} \sum_{y \sim x} \omega_{xy} \phi''(\xi_y) \left( \frac{\theta_1 \lambda_y^{\theta_1 - 1} (d(x_0, y) - d(x_0, x))}{R^{\theta_1}} \right)^2 \\  
        &= \colon I_1 + I_2 + I_3.  
    \end{split}  
    \end{equation*}  
Using the properties of $\phi$ -- namely, that $\phi$ is non-increasing, $\phi'$ is bounded, and $\phi'(\psi_R(x,t))$ vanishes on $E_R^c$ -- along with Assumption \textbf{(A)}, the condition $\theta_1 \geq 2 \geq 1 + \alpha$, and the fact that $(x, t) \in E_R$ implies $d(x_0, x) \leq 2^{\frac{1}{\theta_1}} R$, we obtain the following estimate:
    \begin{equation*}  
        I_1 \leq \frac{C}{R^{1 + \alpha}} 1_{E_R}(x,t).  
    \end{equation*}  
    For $I_2$, note that for every $y \sim x$,  
    \begin{equation}\label{lambda.eta}  
        0 \leq \eta_y, \lambda_y \leq \max\{d(x_0, y), d(x_0, x)\} \leq d(x_0, x) + j.  
    \end{equation}  
    Then, using \textbf{(A)} and the properties of $\phi$ mentioned above, we have for $R>j$  
    \begin{equation*}  
        I_2 \leq \frac{C}{R^2} 1_{E_R}(x,t).  
    \end{equation*}  
   Finally, for $I_3$, it was shown in \cite{dariopara} that for $R \geq 2j$, $\phi''(\xi_y) = 0$ for all $y \sim x$ if $(x,t) \in F_R^c$. By Assumption \textbf{(A)}, \eqref{lambda.eta}, and the fact that $d(x_0, x) + j \leq 6R$ for $(x,t) \in F_R$ and $R \geq 2j$, we conclude for $R \geq 2j$ that
    \begin{equation*}  
        I_3 \leq \frac{C}{R^2} 1_{F_R}(x,t).  
    \end{equation*}  
    Combining the estimates for $I_1$, $I_2$, and $I_3$, we arrive at the claim since $E_R \subset F_R$ and $1 + \alpha \leq 2$. For the proof of (ii), the chain rule and the fact that $\phi'(\psi_R(x,t)) = 0$ on $E_R^c$ yield the estimate.

Let $u$ be a nonnegative global very weak solution of \eqref{maineq}. Observe that $\phi_R(x,t)$ is an admissible test function since its support is contained in  
\begin{equation*}  
   Q_R \coloneqq B_{2^{\frac{1}{\theta_1}} R}(x_0) \times [0, 2^{\frac{1}{\theta_2}} R^{\frac{\theta_1}{\theta_2}}],  
\end{equation*}  
so that all sums for $x \in V$ are finite, and all integrals in the time variable are over compact domains. Additionally, $0 \leq \phi_R \leq 1$ in $V \times [0, +\infty)$, and $\phi(x, \cdot) \in \mathcal{C}^1([0, +\infty))$ for all $x \in V$. Hence, we can test $u$ against $\phi_R^s$ with $s>\max\left(\frac{\sigma}{\sigma-1},\frac{\sigma}{\sigma-m}\right)$, i.e., by \eqref{defneq},  
\begin{equation*}  
\begin{split}  
    \int_0^{\infty} \sum_{x \in V} & \Big[ \Delta(F(u(x,t))) \phi_R^s(x,t) +  v(x,t) u^{\sigma}(x,t) \phi_R^s(x,t) \\  
    &+  u(x,t) \partial_t(\phi^s_R(x,t)) \Big]\;\mu(x)\;dt + \sum_{x \in V}  u(x,0) \phi_R^s(x,0) \;\mu(x) \leq 0.  
\end{split}  
\end{equation*}  
Then, through integration by parts, see $\eqref{ip}$,
\begin{equation*}\label{fin1}  
\begin{split}  
    \int_0^{\infty} \sum_{x \in V} &  v(x,t) u^{\sigma}(x,t) \phi_R^s(x,t) \;\mu(x)\;dt \\  
    &\leq - \int_0^{\infty} \sum_{x \in V}  F(u(x,t)) \Delta(\phi_R^s(x,t)) \;\mu(x)\;dt \\  
    &- \int_0^{\infty} \sum_{x \in V}  u(x,t) \partial_t(\phi^s_R(x,t)) \;\mu(x)\;dt \\  
    &= \colon J_1 + J_2.  
\end{split}  
\end{equation*}  
By the convexity of the map $p \mapsto p^s$ for $p \geq 0$, the bound $F(p) \leq C p^m$ for all $p \in [0, +\infty)$, and \eqref{est1}, we obtain

\begin{equation*}\label{fin2}  
\begin{split}  
    J_1 &\leq -s C \int_0^{\infty} \sum_{x \in V} u^m(x,t) \phi_R^{s-1}(x,t) \Delta(\phi_R(x,t)) \;\mu(x)\;dt \\  
    &\leq \frac{C}{R^{1+\alpha}} \int_0^{\infty} \sum_{x \in V}  u^m(x,t) \phi_R^{s-1}(x,t) 1_{F_R}(x,t)\;\mu(x)\;dt.  
\end{split}  
\end{equation*}  
For $J_2$, by $\eqref{est2}$,  
\begin{equation*}\label{fin3}  
\begin{split}  
    J_2 &= -s \int_0^{\infty} \sum_{x \in V}  u(x,t) \phi^{s-1}_R(x,t) \partial_t(\phi_R(x,t))\;\mu(x)\;dt \\  
    &\leq \frac{C}{R^{\frac{\theta_1}{\theta_2}}} \int_0^{\infty} \sum_{x \in V}  u(x,t) \phi^{s-1}_R(x,t) 1_{E_R}(x,t)\;\mu(x)\;dt.  
\end{split}  
\end{equation*}  
In summary,  
\begin{equation}\label{fin4}  
\begin{split}  
    \int_0^{\infty} \sum_{x \in V} &  v(x,t) u^{\sigma}(x,t) \phi_R^s(x,t) \;\mu(x)\;dt \\  
    &\leq \frac{C}{R^{1+\alpha}} \int_0^{\infty} \sum_{x \in V}  u^m(x,t) \phi_R^{s-1}(x,t) 1_{F_R}(x,t) \;\mu(x)\;dt \\  
    &\;\;\;\;+ \frac{C}{R^{\frac{\theta_1}{\theta_2}}} \int_0^{\infty} \sum_{x \in V}  u(x,t) \phi^{s-1}_R(x,t) 1_{E_R}(x,t) \;\mu(x)\;dt \\  
    &= \colon K_1 + K_2.  
\end{split}  
\end{equation}  
By Young's inequality, since $\sigma > m$,  
\begin{equation*}  
\begin{split}  
    K_1 &\leq \frac{1}{4} \int_0^{\infty} \sum_{x \in V} v(x,t) u^{\sigma}(x,t) \phi_R^s(x,t) \;\mu(x)\;dt \\  
    &+ \frac{C}{R^{(1+\alpha) \frac{\sigma}{\sigma-m}}} \int_0^{\infty} \sum_{x \in V} \phi_R^{s - \frac{\sigma}{\sigma-m}}(x,t) v^{-\frac{m}{\sigma-m}}(x,t) 1_{F_R}(x,t) \;\mu(x)\;dt 
\end{split}  
\end{equation*}  
and, since $\sigma>1$,
\begin{equation*}
\begin{split}
    K_2 \leq& \frac{1}{4} \int_0^{\infty} \sum_{x\in V}  v(x,t)u^{\sigma}(x,t)\phi_R^s(x,t) \;\mu(x)\;dt\\
    &+  \frac{C}{R^{ \frac{\theta_1\sigma}{\theta_2(\sigma-1)}}}  \int_0^{\infty} \sum_{x\in V}  \phi^{s-\frac{\sigma}{\sigma-1}}_R(x,t) v^{-\frac{1}{\sigma-1}} 1_{E_R}(x,t)\;\mu(x)\;dt.
\end{split}
\end{equation*}
Then, for every $R>0$ large enough
\begin{equation*}  
\begin{split}  
    \int_0^{\infty} \sum_{x \in V} &  v(x,t) u^{\sigma}(x,t) \phi_R^s(x,t) \;\mu(x)\;dt \\  
    &\leq \frac{C}{R^{(1+\alpha) \frac{\sigma}{\sigma-m}}} \int_0^{\infty} \sum_{x \in V} v^{-\frac{m}{\sigma-m}}(x,t) 1_{F_R}(x,t) \;\mu(x)\;dt \\  
    & \quad + \frac{C}{R^{\frac{\theta_1 \sigma}{\theta_2 (\sigma-1)}}} \int_0^{\infty} \sum_{x \in V}  v^{-\frac{1}{\sigma-1}}(x,t) 1_{E_R}(x,t) \;\mu(x)\;dt \\   
\end{split}  
\end{equation*}  
Now let $D_R(x_0) \coloneqq \{ (x,t) \in V \times [0, +\infty) : d(x_0, x)^{\theta_1} + t^{\theta_2} \leq R^{\theta_1} \}$ and note that there exists an $l \in \mathbb{N}$, $l \geq 3\theta_1 - 1$, such that for every $R > 0$,
\begin{equation}\label{inclusion}
    F_R \subset \bigcup_{k=0}^l E_{2^{\frac{k}{\theta_1} - 1}R}.  
\end{equation}  
Then, by $\eqref{hp1}$ and $\eqref{hp2}$,
\begin{equation*}  
\begin{split}  
 \int_0^{\infty} \sum_{x \in V} & v(x,t) u^{\sigma}(x,t) 1_{D_R(x_0)}(x,t) \;\mu(x)\;dt \\
& \leq \int_0^{\infty} \smash{\sum_{x \in V}} v(x,t)u^{\sigma}(x,t)\phi_R^s(x,t)\;\mu(x)\;dt \\
& \leq \frac{C}{R^{(1+\alpha) \frac{\sigma}{\sigma-m}}} \sum_{k=0}^l \int_0^{\infty} \sum_{x \in V} v^{-\frac{m}{\sigma-m}}(x,t) 1_{E_{2^{\frac{k}{\theta_1} - 1}R}}(x,t)\;\mu(x)\;dt \\  
& \quad + \frac{C}{R^{\frac{\theta_1 \sigma}{\theta_2 (\sigma-1)}}} \int_0^{\infty} \sum_{x \in V} v^{-\frac{1}{\sigma-1}}(x,t) 1_{E_R}(x,t) \;\mu(x)\;dt \\   
& \leq \frac{C}{R^{(1+\alpha) \frac{\sigma}{\sigma-m}}} \sum_{k=0}^l \left( 2^{\frac{k}{\theta_1}-1} R \right)^{(1+\alpha) \frac{\sigma}{\sigma-m}} + C \\ 
& \leq C.
\end{split}
\end{equation*}
Passing to the limit as $R\rightarrow+\infty$, we have 
\begin{equation}\label{l1} 
    \int_0^{\infty} \sum_{x \in V} v(x,t) u^{\sigma}(x,t) \;\mu(x)\;dt \leq C.  
\end{equation}  
It remains to show that $u \equiv 0$ on $V \times [0, +\infty)$. For this, return to \eqref{fin4}. By Hölder's Inequality and applying $\eqref{inclusion}$ again, we obtain from $\eqref{hp1}$ and $\eqref{hp2}$ that
\begin{equation*}  
\begin{split}  
   \int_0^{\infty} \sum_{x \in V} & v(x,t) u^{\sigma}(x,t) 1_{D_R(x_0)}(x,t)\;\mu(x)\;dt \\  
   &\leq \int_0^{\infty} \sum_{x \in V} v(x,t) u^{\sigma}(x,t) \phi_R^s(x,t) \;\mu(x)\;dt \\  
   &\leq \frac{C}{R^{1+\alpha}} \left( \int_0^{\infty} \sum_{x \in V} \phi_R^{s - \frac{\sigma}{\sigma - m}}(x,t) v^{-\frac{m}{\sigma-m}}(x,t) 1_{F_R}(x,t)\;\mu(x)\;dt \right)^{\frac{\sigma-m}{\sigma}} \\  
   &\quad \times \left( \int_0^{\infty} \sum_{x \in V} v(x,t) u^{\sigma}(x,t) \phi_R^s(x,t) 1_{F_R}(x,t) \;\mu(x)\;dt \right)^{\frac{m}{\sigma}} \\  
   &\quad + \frac{C}{R^{\frac{\theta_1}{\theta_2}}} \left( \int_0^{\infty} \sum_{x \in V}  \phi_R^{s - \frac{\sigma}{\sigma - 1}}(x,t) v^{-\frac{1}{\sigma-1}}(x,t) 1_{E_R}(x,t) \;\mu(x)\;dt \right)^{\frac{\sigma-1}{\sigma}} \\  
   &\quad \times \left( \int_0^{\infty} \sum_{x \in V} v(x,t) u^{\sigma}(x,t) \phi_R^s(x,t) 1_{E_R}(x,t) \;\mu(x)\;dt \right)^{\frac{1}{\sigma}} \\  
   &\leq C \left( \int_0^{\infty} \sum_{x \in V}  v(x,t) u^{\sigma}(x,t) 1_{F_R}(x,t) \;\mu(x)\;dt \right)^{\frac{m}{\sigma}} \\  
   &\quad + C \left( \int_0^{\infty} \sum_{x \in V}  v(x,t) u^{\sigma}(x,t) 1_{E_R}(x,t) \;\mu(x)\;dt \right)^{\frac{1}{\sigma}}.
\end{split}  
\end{equation*}  
Then, letting $R \rightarrow \infty$, by \eqref{l1}, we have  
\begin{equation*}  
     \int_0^{\infty} \sum_{x \in V} v(x,t) u^{\sigma}(x,t) \;\mu(x)\;dt \leq 0.  
\end{equation*}  
Since $\mu$ and $v$ are positive, we conclude that $u \equiv 0$ on $V \times [0, +\infty)$. This completes the proof.  
\end{proof}
In what follows, we will prove the corollaries of our main result.
\begin{proof}[Proof of Corollary~\ref{cor1}]
First, observe that $E_R$ defined in $\eqref{er}$ satisfies  
\begin{equation*}  
    E_R \subset B_{2^{\frac{1}{\theta_1}} R}(x_0) \times \left[ 0, 2^{\frac{1}{\theta_2}} R^{\frac{\theta_1}{\theta_2}} \right].  
\end{equation*}  
For $R$ large enough, by \eqref{cor1a1}, we have 
\begin{equation*}  
\begin{split}  
    \int_0^{\infty} \sum_{x \in V} & v^{-\frac{1}{\sigma-1}}(x,t) 1_{E_R}(x,t) \;\mu(x)\;dt \\  
    &\leq \int_0^{\infty} \sum_{x \in V} f^{-\frac{1}{\sigma-1}}(t) g^{-\frac{1}{\sigma-1}}(x) 1_{E_R}(x,t) \;\mu(x)\;dt \\  
    &\leq \left( \int_0^{2^{\frac{1}{\theta_2}} R^{\frac{\theta_1}{\theta_2}}} f^{-\frac{1}{\sigma-1}}(t) \;dt  \right) \left( \sum_{x \in V} g^{-\frac{1}{\sigma-1}}(x) 1_{B_{2^{\frac{1}{\theta_1}} R}(x_0)} \;\mu(x) \right) \\  
    &\leq C R^{ \frac{\theta_1}{\theta_2} \delta_1 + \delta_2} \leq C R^{\frac{\theta_1 \sigma}{\theta_2 (\sigma-1)}} 
\end{split}  
\end{equation*} 
provided
\begin{equation}\label{cor1pf1}
    \frac{\theta_1}{\theta_2}\delta_1 + \delta_2 \leq \frac{\theta_1}{\theta_2}\frac{\sigma}{\sigma-1}.
\end{equation}
Similarly, applying \eqref{cor1a2},
\begin{equation*}
\begin{split}
      \int_0^{\infty} \sum_{x \in V} & v^{-\frac{m}{\sigma-m}}(x,t) 1_{E_R}(x,t) \;\mu(x)\;dt \leq C R^{(1+\alpha)\frac{\sigma}{\sigma-m}}
\end{split}
\end{equation*}
provided 
\begin{equation}\label{cor1pf2}  
\frac{\theta_1}{\theta_2} \delta_3 + \delta_4 \leq (1+\alpha) \frac{\sigma}{\sigma-m}.
\end{equation}  
By the assumptions on $\delta_1, \delta_2, \delta_3, \delta_4$, we can choose $\theta_1 \geq 2$, $\theta_2 \geq 1$ such that \eqref{cor1pf1} and \eqref{cor1pf2} hold. The result follows from Theorem~\ref{mainthm}. 
\end{proof}  

\begin{proof}[Proof of Corollary~\ref{cor2}]
This is an immediate consequence of Corollary~\ref{cor1}, obtained by choosing $f \equiv 1$ and $\delta_1=\delta_3=1$.
\end{proof}   

\begin{proof}[Proof of Corollary~\ref{cor3}] 
This is an immediate consequence of Corollary~\ref{cor2}, obtained by choosing $g \equiv C'$ and $\delta=\delta'=1$.  
\end{proof}  
 \section{Examples}
In the following example, we consider the particular case of the graph $V=\mathbb{Z}^N$ with $N\geq1$ and $v\equiv1$ as in \cite{dariopara}.
\begin{exm}\label{exm1}
The \textit{$N$-dimensional integer lattice graph} consists of the set of vertices 
\begin{equation*}
\mathbb{Z}^N = \{x = (x_1, \dots, x_N) : x_k \in \mathbb{Z} \text{ for } k = 1, \dots, N\},
\end{equation*}
equipped with the edge weight function 
\begin{equation*}
\omega\colon \mathbb{Z}^N \times \mathbb{Z}^N \to [0, +\infty), \quad \omega_{xy} = 
\begin{cases} 
1, & \text{if } \sum_{k=1}^N |y_k - x_k| = 1, \\
0, & \text{otherwise},
\end{cases}
\end{equation*}
and the node measure 
\begin{equation*}
\mu(x) \coloneqq \sum_{y \sim x} \omega_{xy} = 2N \quad \text{for all } x \in \mathbb{Z}^N.
\end{equation*}
The set of edges is
\begin{equation*}
E = \left\{ (x, y) : x, y \in \mathbb{Z}^N, \, \sum_{i=1}^N |x_i - y_i| = 1 \right\}.
\end{equation*}
The Laplacian of a function $u\colon \mathbb{Z}^N \to \mathbb{R}$ is given by
\begin{equation*}
\Delta u(x) = \frac{1}{2N} \sum_{y \sim x} \left( u(y) - u(x) \right) \quad \text{for all } x \in \mathbb{Z}^N.
\end{equation*}
We endow the graph $(\mathbb{Z}^N, \omega, \mu)$ with the Euclidean distance 
\begin{equation*}
|x - y| \coloneqq d(x, y) = \left( \sum_{i=1}^N (x_i - y_i)^2 \right)^{1/2}.
\end{equation*}
It is straightforward to verify that for $R > 1$, 
\begin{equation*}
\sum_{x \in B_R(0)} \mu(x) = 2N \sum_{x \in B_R(0)} 1 \approx C R^N.
\end{equation*}

Assumption \textbf{(A)} is satisfied with $\alpha=1$. Indeed, by Remark $2.5$ in \cite{dariopara} and Example $5.1$ in \cite{wu2018nonexistence}, for every $x\in \mathbb{Z}^N$ if $d(x,0)\geq1$, we have
\begin{equation*}
    \Delta d(x,0)\leq \frac{1}{2 d(x,0)}.
\end{equation*}
By Corollary~\ref{cor3}, we conclude that the problem 
\begin{equation*}
    u_t \geq \Delta(F(u)) + u^\sigma \quad \text{in } \mathbb{Z}^N \times [0, +\infty),
\end{equation*}
with $m \in (0, +\infty)$ fixed and $F\colon[0,+\infty)\to[0,+\infty)$ such that $0\leq F(p) \leq C p^m$ for all $p\geq 0$, does not admit a nontrivial, nonnegative solution if 
\begin{equation*}
    \sigma > \max(1, m) \quad \text{and} \quad N \leq \frac{2}{\sigma - m}, \quad \text{i.e., } \max(1, m) < \sigma \leq  m + \frac{2}{N}.
\end{equation*}
In particular, $m>1-\frac{2}{N}$ in the nonexistence regime.

By the global existence result in \cite{lin2017existence}, we know that 
\begin{equation*}
    u_t \geq \Delta u + u^\sigma \quad \text{in } \mathbb{Z}^N \times (0, +\infty)
\end{equation*}
admits a global-in-time nonnegative solution if the initial condition $u(x, 0) = u_0(x)$ is sufficiently small and $\sigma > 1 + \frac{2}{N}$. Therefore, in the case $F(u) = u$, the result in Theorem~\ref{mainthm} is optimal, as also observed in \cite{dariopara}.

Though it is not known whether our result is optimal in the context of nonlinear reaction-diffusion equations on graphs, we comment that in $\mathbb{R}^N$ the equation
\begin{equation*}
    u_t = \Delta u^m + u^\sigma,
\end{equation*}
with $m > 1$ and $\sigma \in \big(1, m + \frac{2}{N}\big)$, does not admit nonnegative global solutions for any initial datum, see \cite[Theorem~2, p.\,217]{SGKM}, whereas for $\sigma > m + \frac{2}{N}$ small positive data give rise to global solutions. Though we shall not discuss this issue here, we expect the result to be sharp at least on $\mathbb{Z}^N$ in the Porous Medium case.
\end{exm}
As we will see below, one has the same range of nonexistence when pairing the integer lattice with a finite group.
\begin{exm}
In \cite{dariopara}, it was shown that if $(V_1, \omega_1, \mu_1)$ is an infinite graph satisfying Assumption \textbf{(A)} for some $\alpha \in [0, 1]$ and $(V_2, \omega_2, \mu_2)$ is a finite graph, then the product graph $(V, \omega, \mu)$, defined by $V = V_1 \times V_2$, also satisfies Assumption \textbf{(A)} for the same $\alpha$. The edge weights $\omega$ and measure $\mu$ are defined as follows: for $(x_1, x_2), (y_1, y_2) \in V_1 \times V_2$,
\begin{equation*}
\omega((x_1, x_2), (y_1, y_2)) := 
\begin{cases}
\omega_1(x_1, y_1) & \text{if } x_2 = y_2, \\
\omega_2(x_2, y_2) & \text{if } x_1 = y_1, \\
0 & \text{otherwise},
\end{cases}
\end{equation*}
and $\mu$ is a positive measure on $V_1 \times V_2$ satisfying
\begin{equation*}
\mu(x_1, x_2) \geq C \max\{\mu_1(x_1), \mu_2(x_2)\}
\end{equation*}
for $(x_1, x_2) \in V_1 \times V_2$. For $\alpha$ as above, a pseudometric $d$ on $V$ is introduced by
\begin{equation*}
d((x_1, x_2), (y_1, y_2)) := \left(d_1^{1+\alpha}(x_1, y_1) + d_2^{1+\alpha}(x_2, y_2)\right)^{1/(1+\alpha)},
\end{equation*}
where $d_1$ and $d_2$ are pseudometrics on $V_1$ and $V_2$, respectively, with finite jump sizes $j_1$ and $j_2$. The resulting pseudometric $d$ has a finite jump size $j \leq \max\{j_1, j_2\}$. Additionally, if both $V_1$ and $V_2$ are connected, locally finite and undirected, then so is $V$.

A particular example is given by taking $(V_1, \omega_1, \mu_1)$ as in Example~\ref{exm1} and $(V_2, \omega_2, \mu_2)$ as follows: let $V_2$ be a finite group $(G,\cdot)$ generated by $h \in G$, i.e. $G$ is cyclic. Define the edge weights $\omega_2$ by
\begin{equation*}
\omega_2(g, g') := 
\begin{cases}
1 & \text{if } g' = g \cdot h \text{ or } g' = g \cdot h^{-1}, \\
0 & \text{otherwise},
\end{cases}
\end{equation*}
for $g, g' \in G$, and the measure $\mu_2$ by
\begin{equation*}
\mu_2(g) := \sum_{g' \in G} \omega_2(g, g').
\end{equation*}
Equip $V_2$ with the natural distance as defined in \eqref{natdist}.

For $V = \mathbb{Z}^N \times G^K$ with $K \geq 1$, it is easy to verify that for $R > 0$ sufficiently large, the volume growth satisfies
\begin{equation*}
\text{Vol}\left(B_R((0,h ))\right) \leq C R^N.
\end{equation*}
By Corollary~\ref{cor3}, the problem
\begin{equation*}
u_t \geq \Delta(F(u)) + u^\sigma \quad \text{in } \mathbb{Z}^N \times G^K \times [0, +\infty),
\end{equation*}
with $m \in (0, +\infty)$ fixed and $F: [0, +\infty) \to [0, +\infty)$ such that $0 \leq F(p) \leq C p^m$ for all $p \geq 0$, does not admit a nontrivial, nonnegative solution if
\begin{equation*}
\sigma > \max(1, m) \quad \text{and} \quad N \leq \frac{2}{\sigma - m}, \quad \text{i.e., } \max(1, m) < \sigma \leq  m + \frac{2}{N}.
\end{equation*}
This coincides with the bounds in Example~\ref{exm1}. Hence, it is independent of the specific group $G$.
\end{exm}
We now present an example involving a graph with exponential volume growth.
\begin{exm}
A weighted graph $(T, \omega, \mu)$ is a \textit{tree} if, for each $x,y\in T$, there exists a unique minimal path connecting $x$ to $y$. Here, the set of edges $E$ is determined by declaring which pairs $x, y \in T$ satisfy $y \sim x$, rather than explicitly specifying a weight function $\omega$. We assume that a suitable weight function $\omega$ exists, inducing the set of edges $E$.

Equipping $T$ with the natural distance $d_*$ defined in \eqref{natdist}, we define, for $k \geq 0$ and a fixed node $x_0\in T$,
\begin{equation}\label{dk}
D_k \coloneqq \{ x \in T : d_*(x_0, x) = k \}.
\end{equation}
If $x \in D_k$ for $k \geq 1$, there exists a unique $y \in D_{k-1}$ such that $y \sim x$.

A tree is \textit{homogeneous of degree $N > 1$} if:
\begin{itemize}
    \item Every node $x \in D_k$ for $k \geq 1$ has exactly one neighbour $y \in D_{k-1}$ and exactly $N$ neighbours $y_1, \dots, y_N \in D_{k+1}$.
    \item The node $x_0$ has exactly $N+1$ neighbours $y_1, \dots, y_{N+1} \in D_1$. We call $x_0$ the root of $T$.
\end{itemize}
Now consider the problem in \eqref{maineq} with $m\in(0,+\infty)$ fixed and $\sigma>\max(1,m)$. Suppose there exists a constant $C>0$ such that $0\leq F(p)\leq Cp^m$ for all $p\geq0$. We equip $T$ with the following edge weights:
\begin{equation*}
\omega_{xy} = 
\begin{cases}
\frac{1}{N+1} & \text{if } x \sim y, \\
0 & \text{otherwise},
\end{cases}
\end{equation*}
and node measure $\mu(x) \coloneqq \text{degree}(x) = N+1$. Under this construction, $T$ is a locally finite, connected, undirected infinite graph. In this setting, we equip $T$ with the natural distance $d_*$, and by Remark~\ref{natdistrmk}, Assumption \textbf{(A)} holds with $\alpha = 0$.

Let $v(x,t) = g(x)$ be a time-independent potential, where
\begin{equation*}
g(x) \geq Cd_*(x,x_0)^{\lambda} N^{\lambda d_*(x,x_0)}\quad\text{for } \lambda=\max\left( \sigma-1, \frac{\sigma-m}{m} \right),
\end{equation*}
i.e., $g(x)$ grows exponentially with distance from $x_0$. For any $R > 1$, since the cardinality of the sets $D_k$ defined in \eqref{dk} is $|D_k| = N^k + N^{k-1}$,
\begin{equation*}
\begin{split}
\sum_{x \in B_R} g(x)^{-\frac{1}{\sigma-1}} \;\mu(x) &= \sum_{k=1}^{\lfloor R \rfloor} \sum_{x \in D_k} g(x)^{-\frac{1}{\sigma-1}} \;\mu(x) + g(x_0)^{-\frac{1}{\sigma-1}} \mu(x_0) \\
&\leq C(N+1) \sum_{k=1}^{\lfloor R \rfloor} |D_k| \frac{1}{kN^k} + C \\
&\leq C \sum_{k=1}^{\lfloor R \rfloor} \frac{1}{k}+C\\
&\leq C\log R.
\end{split}
\end{equation*}
Similarly,
\begin{equation*}
\sum_{x \in B_R} g(x)^{-\frac{m}{\sigma-m}} \;\mu(x) \leq C\log R.
\end{equation*}

By Corollary~\ref{cor2}, there are no nontrivial, nonnegative solutions to \eqref{maineq} for all $\sigma>\max(1,m)$.
\end{exm}

\section{Finite Graphs}
On finite graphs, Assumption \textbf{(A)} is inherently satisfied by $d_*$ (see $\eqref{natdist}$ for a definition). In the following theorem, we explore the nonexistence of solutions to the problem in \eqref{maineq} in this context, with a generic $F \colon [0,+\infty)\to \mathbb{R}$.

\begin{rmk}
Graphs do not admit a topological definition of a boundary. However, for a graph $V$ and a subgraph $G \subset V$, the boundary of $G$ relative to $V$ can be defined as  
\begin{equation*}
\partial G = \{ x \in G : \exists \, y \in V \setminus G, \, y \sim x \}.  
\end{equation*}
In Theorem~\ref{finigraphthm}, we consider finite graphs, which have no boundary, analogous to compact Riemannian manifolds without boundary. For $\eqref{maineq}$ on a finite subgraph $G \subset V$ with appropriate boundary conditions, the results may differ. This is the analogue of the case of compact Riemannian manifolds with boundary.
\end{rmk}

\begin{thm}\label{finigraphthm} 
Let $(V, \omega, \mu)$ be a weighted finite graph. Let $\sigma > 1$, $v\colon V \times [0, +\infty) \to \mathbb{R}$ be a positive function, and $F\colon\mathbb{R}\to\mathbb{R}$. Suppose that for every $T \geq T_0 > 0$,  
\begin{equation}\label{finitegrhp}
    \int_T^{2T} \sum_{x \in V} v^{-\frac{1}{\sigma-1}}(x,t) \;\mu(x)\;dt \leq C T^{\frac{\sigma}{\sigma-1}}.  
\end{equation}  
Let $u\colon V \times [0, +\infty) \to \mathbb{R}$ be a nonnegative global very weak solution of \eqref{maineq}. Then, $u \equiv 0$.  
\end{thm}  

\begin{proof}  
The proof follows the same approach as Theorem 5.2 in \cite{dariopara}. For completeness, we provide the details: Let $\varphi \in C^2([0,+\infty))$ satisfy $\varphi \equiv 1$ on $[0, 1]$, $\varphi \equiv 0$ on $[2,+\infty)$, and $\varphi' \leq 0$. Define $\phi_T(t) \coloneqq \varphi\left(\frac{t}{T}\right)$ for $T > 0$. Since $V$ is finite, $\phi_T$ is an admissible test function. Hence, we can test $u$ against $\phi_T^s$ with $s > \frac{\sigma}{\sigma-1}$. Observe that, by $\eqref{ip}$,
\begin{equation*}  
    \sum_{x \in V}  \Delta \left(F(u(x,t))\right) \phi_T^s(t) \;\mu(x) = \sum_{x \in V} F(u(x,t)) \Delta \phi_T^s(t) \;\mu(x) = 0.  
\end{equation*}  
From $\eqref{defneq}$, it follows that
\begin{equation}\label{fingrhölder}
\begin{split}
\int_0^{\infty} \sum_{x \in V} v(x, t) u^\sigma(x, t) \phi_T^s(t) \;\mu(x)\; dt &\leq -s \int_0^{\infty} \sum_{x \in V} u(x, t) \phi_T^{s-1}(t) \frac{\partial \phi_T}{\partial t}(t)\;\mu(x)\; dt\\
&\leq \frac{C}{T} \int_T^{2T} \sum_{x \in V} u(x, t) \phi_T^{s-1}(t) \;\mu(x)\;dt.
\end{split}
\end{equation}
By Young’s Inequality and $\eqref{finitegrhp}$, we have
\begin{equation*}
\int_0^{\infty} \sum_{x \in V} v(x, t) u^\sigma(x, t) \phi_T^s(t) \;\mu(x)\; dt \leq \frac{C}{T^{\frac{\sigma}{\sigma-1}}} \int_T
^{2T} \sum_{x \in V} v(x, t)^{-\frac{1}{\sigma-1}} \;\mu(x)\; dt \leq C.
\end{equation*}
Thus,
\begin{equation*}
   \int_T^{2T} \sum_{x \in V} v(x, t) u^\sigma(x, t) \;\mu(x)\; dt  \leq \int_0^{\infty} \sum_{x \in V} v(x, t) u^\sigma(x, t) \phi_T^s(t) \;\mu(x)\; dt \leq C
\end{equation*}
and we conclude 
\begin{equation}\label{fingrfinest}
    \int_0^{+\infty} \sum_{x \in V} v(x, t) u^\sigma(x, t) \;\mu(x)\; dt \leq C. 
\end{equation}
It remains to show that $u \equiv 0$ on $V \times [0,+\infty)$. For this, return to $\eqref{fingrhölder}$. Applying Hölder’s inequality and $\eqref{finitegrhp}$, we obtain
\begin{equation*}
\begin{split}
\int_0^T \sum_{x \in V} v(x, t) u^\sigma(x, t) \;\mu(x)\; dt &\leq \int_0^{\infty} \sum_{x \in V} v(x, t) u^\sigma(x, t) \phi_T^s(t) \;\mu(x)\; dt\\
&\leq C \left( \int_T^{2T} \sum_{x \in V} v(x, t) u^\sigma(x, t) \;\mu(x)\; dt \right)^{\frac{1}{\sigma}}.
\end{split}
\end{equation*}
Letting $T \to \infty$ and using $\eqref{fingrfinest}$, we conclude
\begin{equation*}
\int_0^{+\infty} \sum_{x \in V} v(x, t) u^\sigma(x, t) \;\mu(x)\; dt = 0.
\end{equation*}
Hence, $u \equiv 0$ on $V \times [0,+\infty)$.
\end{proof}  

An immediate consequence is the following result for a time-independent potential $v$; see also \cite{wu2018nonexistence} for the case $v\equiv1$:  
\begin{cor}  
Let $(V, \omega, \mu)$ be a weighted finite graph. Let $\sigma > 1$, $v\colon V \to \mathbb{R}$ be a positive function, and $F\colon\mathbb{R}\to\mathbb{R}$. Let $u\colon V \times [0, +\infty) \to \mathbb{R}$ be a nonnegative global very weak solution of \eqref{maineq}. Then, $u \equiv 0$.  
\end{cor}  


\end{document}